\documentclass[10pt]{article}

\usepackage{amsmath}
\usepackage{amsthm}
\usepackage{amssymb}
\usepackage{amscd}
\usepackage{amsfonts}  
\usepackage{makeidx}
\usepackage{hyperref}
\usepackage[all]{xy} \SelectTips{cm}{}
\usepackage{makeidx}     
\usepackage{graphicx}    
\usepackage{multicol}    

\DeclareMathAlphabet\EuR{U}{eur}{m}{n}
\SetMathAlphabet\EuR{bold}{U}{eur}{b}{n}

\makeindex             

\theoremstyle{plain}
\newtheorem{theorem}{Theorem}[section]
\newtheorem{lemma}[theorem]{Lemma}

\newtheorem{corollary}[theorem]{Corollary}

\theoremstyle{definition}
\newtheorem{definition}[theorem]{Definition}

\newtheorem{condition}[theorem]{Condition}

{\catcode`@=11\global\let\c@equation=\c@theorem}

\newcommand{\comsquare}[8]                   
{\begin{CD}
#1 @>#2>> #3\\
@V{#4}VV @V{#5}VV\\
#6 @>#7>> #8
\end{CD}
}

\newcommand{\xycomsquare}[8]                   
{\xymatrix
{#1 \ar[r]^{#2} \ar[d]^{#4} &
#3 \ar[d]^{#5}  \\
#6\ar[r]^{#7} &
#8
}
}

\newcommand{\calfin}{\mathcal{FIN}}

\newcommand{\calicyc}{\mathcal{ICY}}

\newcommand{\calvcyc}{\mathcal{VCY}}
\newcommand{\caltr}{\{ \! 1 \! \}}

\newcommand{\calall}{{\cala\!\call\!\call}}

\newcommand{\cala}{{\cal A}}

\newcommand{\calf}{{\cal F}}
\newcommand{\calg}{{\cal G}}

\newcommand{\call}{{\cal L}}

\newcommand{\caln}{{\cal N}}

\newcommand{\calu}{{\cal U}}
\newcommand{\calv}{{\cal V}}

\newcommand{\IQ}{{\mathbb Q}}
\newcommand{\IR}{{\mathbb R}}

\newcommand{\IZ}{{\mathbb Z}}

\newcommand{\curs}{\EuR}

\newcommand{\Or}{\curs{Or}}

\newcommand{\CAT}{\operatorname{CAT}}
\newcommand{\cd}{\operatorname{cd}}

\newcommand{\colim}{\operatorname{colim}}

\newcommand{\cyl}{\operatorname{cyl}}

\newcommand{\Ext}{\operatorname{Ext}}

\newcommand{\hdim}{\operatorname{hdim}}

\newcommand{\id}{\operatorname{id}}

\newcommand{\Min}{\operatorname{Min}}

\newcommand{\pr}{\operatorname{pr}}

\newcommand{\res}{\operatorname{res}}

\newcommand{\supp}{\operatorname{supp}}

\newcommand{\topdim}{\operatorname{top-dim}}

\newcommand{\EGF}[2]{E_{#2}(#1)}                   
\newcommand{\eub}[1]{\underline{E}#1}              
\newcommand{\jub}[1]{\underline{J}#1}              
\newcommand{\edub}[1]{\underline{\underline{E}}#1} 

\newcommand{\higherlim}[3]{{\setbox1=\hbox{\rm lim}
        \setbox2=\hbox to \wd1{\leftarrowfill} \ht2=0pt \dp2=-1pt
        \mathop{\vtop{\baselineskip=5pt\box1\box2}}
        _{#1}}^{#2}#3}

\newcommand{\version}[1]                       
{\begin{center} last edited on #1\\
last compiled on \today\\
name of texfile: \jobname
\end{center}
}

\newcounter{commentcounter}

\begin{document}

\typeout{----------------------------  catvcyc.tex  ----------------------------}

\title{On the classifying space of the family of finite and of virtually cyclic subgroups
for $\CAT(0)$-groups}
\author{Wolfgang L\"uck\thanks{\noindent email:
lueck@math.uni-muenster.de \protect\\
www:~http://www.math.uni-muenster.de/u/lueck/\protect\\
FAX: 49 251 8338370\protect}\\
Fachbereich Mathematik\\ Universit\"at M\"unster\\
Einsteinstr.~62\\ 48149 M\"unster\\Germany}
\maketitle


\typeout{-----------------------  Abstract  ------------------------}

\begin{abstract}
  Let $G$ be a discrete group which acts properly and isometrically on
  a complete $\CAT(0)$-space $X$.  Consider an integer $d$ with 
  $d =  1$ or $d \ge 3$ such that the topological dimension of $X$ is
  bounded by $d$. We show the existence of a $G$-$CW$-model $\eub{G}$
  for the classifying space for proper $G$-actions with $\dim(\eub{G}) \le d$.  
  Provided that the action is also cocompact, we prove the
  existence of a $G$-$CW$-model $\edub{G}$ for the classifying space
  of the family of virtually cyclic subgroups satisfying
  $\dim(\edub{G}) \le d+1$.
  \\[2mm]
\end{abstract}


 \typeout{--------------------   Section 0: Introduction --------------------------}

\setcounter{section}{-1}
\section{Introduction}

Given a group $G$, denote by $\eub{G}$ a $G$-$CW$-model for the
classifying space for proper $G$-actions and by 
$\edub{G} = \EGF{G}{\calvcyc}$ a $G$-$CW$-model for the classifying space of the
family of virtually cyclic subgroups. Our main theorem which will be proven in
Section~\ref{sec:The_passage_from_finite_to_virtually_cyclic_groups} is

\begin{theorem} \label{the:dim(edub(G))_for_Cat(0)-group} Let $G$ be a
  discrete group which acts properly and isometrically on a complete
  proper $\CAT(0)$-space $X$.  Let $\topdim(X)$ be the topological
  dimension of $X$.  Let $d$ be an integer satisfying $d = 1$ or 
  $d \ge 3$ such that $\topdim(X) \le d$.

  \begin{enumerate}

  \item \label{the:dim(edub(G))_for_Cat(0)-group:fin} Then there is
    $G$-$CW$-model $\eub{G}$ with $\dim(\eub{G}) \le d$;

  \item \label{the:dim(edub(G))_for_Cat(0)-group:vcyc} Suppose that
    $G$ acts by semisimple isometries.  (This is the case if we
    additionally assume that the $G$-action is cocompact.)

    Then there is $G$-$CW$-model $\edub{G}$ with $\dim(\edub{G}) \le d+1$.
  \end{enumerate}
\end{theorem}

There is the question whether for any group $G$ the inequality
\begin{equation}
  \hdim^{G}(\eub{G}) - 1 \le \hdim^{G}(\edub{G}) \le \hdim^{G}(\eub{G}) + 1  
  \label{conj_inequality}
\end{equation}
holds, where $\hdim^{G}(\eub{G})$ is the minimum of the dimensions of
all possible $G$-$CW$-models for $\eub{G}$ and $\hdim^{G}(\edub{G})$
is defined analogously
(see~\cite[Introduction]{Lueck-Weiermann(2007)}).
Since $\hdim(\eub{G})
\le 1 + \hdim(\edub{G})$ holds for all groups $G$
(see~\cite[Corollary~5.4]{Lueck-Weiermann(2007)}), 
Theorem~\ref{the:dim(edub(G))_for_Cat(0)-group} implies

\begin{corollary} \label{cor:inequality} Let $G$ be a discrete group
  and let $X$ be complete $\CAT(0)$-space $X$ with finite topological
  dimension $\topdim(X)$. Suppose that $G$ acts properly and
  isometrically on $X$. Assume that the $G$-action is by semisimple
  isometries. (The last condition is automatically satisfied if we
  additionally assume that the $G$-action is cocompact.)  Suppose that
  $\topdim(X) = \hdim^{G}(\eub{G}) \not= 2$.

  Then inequality~\eqref{conj_inequality} is true.
\end{corollary}

We will prove at the end of 
Section~\ref{sec:The_passage_from_finite_to_virtually_cyclic_groups}

\begin{corollary} \label{cor:manifold} Suppose that $G$ is virtually
  torsionfree. Let $M$ be a simply connected complete Riemannian
  manifold of dimension $n$ with non-negative sectional curvature.
  Suppose that $G$ acts on $M$ properly, isometrically and
  cocompactly. Then
$$\begin{array}{rcccl}
  & & \hdim(\eub{G}) & = & n;
  \\
  n-1 & \le & \hdim(\edub{G}) & \le & n+1.
\end{array}
$$
In particular~\eqref{conj_inequality} holds.
\end{corollary}

If $G$ is the fundamental group of an $n$-dimensional closed
hyperbolic manifold, then $\hdim(\eub{G}) = \hdim(\edub{G}) = n$
by~\cite[Example~5.12]{Lueck-Weiermann(2007)}.  If $G$ is virtually
$\IZ^n$ for $n \ge 2$, then $\hdim(\eub{G}) = n$ and $\hdim(\edub{G})= n+1$ 
by~\cite[Example~5.21]{Lueck-Weiermann(2007)}. Hence the cases
$\hdim(\edub{G}) = \hdim(\eub{G})$ and 
$\hdim(\edub{G}) = \hdim(\eub{G}) + 1$ do occur in the situation of
Corollary~\ref{cor:manifold}.  There exists groups $G$ with
$\hdim(\edub{G}) = \hdim(\eub{G}) - 1$
(see~\cite[Example~5.29]{Lueck-Weiermann(2007)}). But we do not
believe that this is possible in the situation of
Corollary~\ref{cor:inequality} or Corollary~\ref{cor:manifold}.

The paper was supported by the Sonder\-forschungs\-be\-reich
478 -- Geometrische Strukturen in der Mathematik -- and the
Max-Planck-Forschungspreis and the Leibniz-Preis of the author.


\typeout{--------------------   Section 1:   --------------------------}

\section{Classifying Spaces for Families}
\label{sec:Classifying_Spaces_for_Families}

We briefly recall the notions of a family of subgroups and the associated
classifying space. For more information, we refer for instance to the 
original source~\cite{Dieck(1972)} or to the  survey article~\cite{Lueck(2005s)}.

A \emph{family $\calf$ of subgroups} of $G$ is a set of  subgroups of $G$ which is
closed under conjugation and taking subgroups. Examples for $\calf$ are
\begin{align*}
    \caltr    & = \{\text{trivial subgroup}\}; \\
    \calfin   & = \{\text{finite subgroups}\}; \\
    \calvcyc  & = \{\text{virtually cyclic subgroups}\};\\
    \calall   & = \{\text{all subgroups}\}.
\end{align*}

Let $\calf$ be a family of subgroups of $G$. A model for the \emph{classifying
space $\EGF{G}{\calf}$ of the family $\calf$} is a $G$-$CW$-complex $X$ all of
whose isotropy groups belong to $\calf$ such that for any $G$-$CW$-complex
$Y$ with isotropy groups in $\calf$ there exists a $G$-map $Y \to X$
and any two $G$-maps $Y \to X$ are $G$-homotopic. 
In other words, $X$ is a terminal object in the
$G$-homotopy category of $G$-$CW$-complexes whose isotropy groups belong to
$\calf$. In particular, two models for $\EGF{G}{\calf}$ are $G$-homotopy
equivalent. 

There exists a model for $\EGF{G}{\calf}$ for any group $G$ and any family $\calf$ of
subgroups. There is even a functorial construction 
(see~\cite[page~223 and Lemma~7.6 (ii)]{Davis-Lueck(1998)}).

A $G$-$CW$-complex $X$ is a model for $\EGF{G}{\calf}$ if and only if the
$H$-fixed point set $X^H$ is contractible for $H \in \calf$ and is empty for $H
\not\in \calf$.

We abbreviate $\eub{G} := \EGF{G}{\calfin}$ and call it the \emph{universal
$G$-$CW$-complex for proper $G$-actions}. We also abbreviate $\edub{G} :=
\EGF{G}{\calvcyc}$.

A model for $\EGF{G}{\calall}$ is $G/G$. A model for $\EGF{G}{\caltr}$ is the
same as a model for $EG$, which denotes the total space of the universal
$G$-principal bundle $EG \to BG$. 

One can also define a numerable version of the space for proper
$G$-actions to $G$ which is denoted by $\jub{G}$. 
It is not necessarily a $G$-$CW$-complex. A metric space $X$ on which 
$G$ acts isometrically and properly is a model for $\jub{G}$ if and only if
the two projections $X \times X \to X$ onto the first  and second factor
are $G$-homotopic to one another.
If $X$ is a complete $\CAT(0)$-space on which $G$-acts properly
and isometrically, then $X$ is a model for $\jub{G}$, the desired $G$-homotopy
is constructed using the geodesics joining two points in $X$ 
(see~\cite[Proposition~1.4 in~II.1 on page~160]{Bridson-Haefliger(1999)}).

One motivation for studying the spaces $\eub{G}$ and $\edub{G}$ comes from the 
Baum-Connes Conjecture and the Farrell-Jones Conjecture.


\typeout{-------------------- Section 2 --------------------------}

\section{Topological and $CW$-dimension}
\label{sec:Topological_and_CW-dimension}

Let $X$ be a topological space. Let $\calu$ be an open covering.
Its \emph{dimension} $\dim(\calu) \in \{0,1,2, \ldots \} \amalg \{\infty\}$ 
is the infimum over all integers $d \ge 0$
such that for any collection $U_0$, $U_1$, \ldots , $U_d$ of pairwise
distinct elements in $\calu$ the intersection $\bigcap_{i=0}^d U_i$ is empty.
An open covering $\calv$ is a refinement of $\calu$ if for every $V \in \calv$ there
is $U \in \calu$ with $V \subseteq U$. 

\begin{definition}[Topological dimension]
  The \emph{topological dimension} of a topological space $X$
$$\topdim(X) \in \{0,1,2, \ldots \} \amalg \{\infty\}$$ 
is the infimum over all integers $d \ge 0$ such that any open covering
$\calu$ possesses a refinement $\calv$ with $\dim(\calv) \le d$.
\end{definition}

Let $Z$ be a metric space. We will denote for $z \in Z$ and $r \ge 0$
by $B_r(z)$ and $\overline{B}_r(z)$ respectively the \emph{open ball}
and \emph{closed ball} respectively around $z$ with radius $r$.  We
call $Z$ \emph{proper} if for each $z \in Z$ and $r \ge 0$ the closed
ball $\overline{B}_r(z)$ is compact.  A group $G$ acts \emph{properly}
on the topological space $Z$ if for any $z \in Z$ there is an open
neighborhood $U$ such that the set $\{g \in G \mid g\cdot U \cap U
\not= \emptyset\}$ is finite.  In particular every isotropy group is
finite. If $Z$ is a $G$-$CW$-complex, then $Z$ is a proper $G$-space
if and only if the isotropy group of any point in $Z$ is finite
(see~\cite[Theorem 1.23]{Lueck(1989)}).

\begin{lemma} \label{dim(Z/G)_le_dim(Z)} Let $Z$ be a proper metric
  space.  Suppose that $G$ acts on $Z$ isometrically and properly.
  Then we get for the topological dimensions of $X$ and $G\backslash
  X$
$$\topdim(G\backslash X) \le \topdim(X).$$
\end{lemma}
\begin{proof}
  Since $G$ acts properly and isometrically, we can find for every 
  $z \in Z$ a real number $\epsilon(z) > 0$ such that we have for all $g \in G$
$$
g \cdot B_{7\epsilon(z)}(z) \cap B_{7\epsilon}(z) \not= \emptyset
\Longleftrightarrow g \cdot B_{7\epsilon(z)}(z) = B_{7\epsilon(z)}(z)
\Longleftrightarrow g \in G_z.
$$
We can arrange that $\epsilon(gz) = \epsilon(z)$ holds for $z \in Z$
and $g \in G$.  Consider $G \cdot \overline{B}_{\epsilon}(z)$. We
claim that this set is closed in $Z$.  We have to show for a sequence
$(z_n)_{n\ge 0}$ of elements in $\overline{B}_{\epsilon}(z)$ and
$(g_n)_{n \ge 0}$ of elements in $G$ and $x \in Z$ with 
$\lim_{n \to \infty} g_n z_n = x$ that $x$ belongs to $G \cdot
\overline{B}_{\epsilon}(z)$. Since $X$ is proper, we can find 
$y \in \overline{B}_{\epsilon}(z)$ such that $\lim_{n \to \infty} z_n = y$.
Choose $N = N(\epsilon)$ such that $d_X(g_nz_n,x) \le \epsilon$ and
$d_X(z_n,y) \le \epsilon$ holds for $n \ge N$.  We conclude for $n \ge N$
\begin{eqnarray*}
  d_x(g_ny,x) 
  & \le & 
  d_X(g_ny,g_nz_n) + d_X(g_nz_n,x)
  \\
  & = & 
  d_X(y,z_n) + d_X(g_nz_n,x)
  \\
  & \le &
  \epsilon +\epsilon
  \\ 
  & = & 
  2\epsilon.
\end{eqnarray*}
This implies for $n \ge N$
\begin{eqnarray*}
  d_X(g_n^{-1}g_Nz,z) 
  & = & 
  d_X(g_Nz,g_nz) 
  \\
  & \le & 
  d_X(g_Nz,g_Ny) + d_X(g_Ny,x) + d_X(x,g_ny) + d_X(g_ny,g_nz)
  \\
  & = & 
  d_X(z,y) + d_X(g_Ny,x) + d_X(g_ny,x) + d_X(y,z)
  \\
  & \le & 
  \epsilon + 2\epsilon + 2\epsilon + \epsilon
  \\
  & = & 
  6 \epsilon.
\end{eqnarray*}
Hence $g_n^{-1}g_N \in G_z$ for $n \ge N$. Since $G_z$ is finite, we
can arrange by passing to subsequences that $g_0 = g_n$ holds for $n
\ge 0$. Hence
$$x = \lim_{n \to \infty} g_n z_n =  \lim_{n \to \infty} g_0  z_n
= g_0 \cdot \lim_{n \to \infty} z_n = g_0 \cdot y \in G\cdot
\overline{B}_{\epsilon}(z).$$ Choose a set-theoretic section $s \colon
G/G_z \to G$ of the projection $G \to G/G_z$. The map
$$G/G_z \times B_{7\epsilon(z)}(z) \xrightarrow{\cong} G \cdot B_{7\epsilon(z)}(z), 
\quad (gG_z,x) \mapsto s(gG_z)\cdot x$$ is bijective, continuous and open
and hence a homeomorphism. It induces a homeomorphism
$$G/G_z \times \overline{B}_{\epsilon(z)}(z) 
\xrightarrow{\cong} G \cdot \overline{B}_{\epsilon(z)}(z).$$
This implies
\begin{eqnarray}
  \topdim\bigl(\overline{B}_{\epsilon(z)}(z)\bigr) 
  & = & 
  \topdim\bigl(G \cdot \overline{B}_{\epsilon(z)}(z)\bigr).
  \label{dim(Z/G)_le_dim(Z):dim(B)_is_dim(gB)}
\end{eqnarray}

Let $\pr \colon Z \to G\backslash Z$ be the projection.  It induces a
bijective continuous map $G_z\backslash \overline{B}_{\epsilon(z)}(z)
\xrightarrow{\cong} \pr\bigl(\overline{B}_{\epsilon(z)}(z)\bigr)$
which is a homeomorphism since $\overline{B}_{\epsilon(z)}(z)$ and
hence $G_z\backslash \overline{B}_{\epsilon(z)}(z)$ is compact.  Hence
we get
\begin{eqnarray}
  \topdim\bigl(\pr(\overline{B}_{\epsilon(z)}(z))\bigr) 
  & = &
  \topdim\bigl(G_z\backslash \overline{B}_{\epsilon(z)}(z)\bigr).
  \label{dim(Z/G)_le_dim(Z):dim(pr(B))_is_dim(B/G_z)}
\end{eqnarray}
Since the metric space $\overline{B}_{\epsilon(z)}(z)$ is compact and
hence contains a countable dense set and $G_z$ is finite, we conclude
from~\cite[Exercise~in Chapter~II on page~112]{Bredon(1972)}
\begin{eqnarray}
  \topdim\bigl(G_z\backslash \overline{B}_{\epsilon(z)}(z)\bigr) 
  & \le &
  \topdim\bigl(\overline{B}_{\epsilon(z)}(z)\bigr).
  \label{dim(Z/G)_le_dim(Z):dim(B/G_z)_le_dim(B)}
\end{eqnarray}
From~\eqref{dim(Z/G)_le_dim(Z):dim(B)_is_dim(gB)},
\eqref{dim(Z/G)_le_dim(Z):dim(pr(B))_is_dim(B/G_z)}
and~\eqref{dim(Z/G)_le_dim(Z):dim(B/G_z)_le_dim(B)} we conclude that
$G \cdot \overline{B}_{\epsilon(z)}(z) \subseteq Z$ and
$\pr\bigl(\overline{B}_{\epsilon(z)}(z)\bigr) \subseteq G\backslash Z$
are closed and satisfy
\begin{eqnarray}
  \topdim\bigl(\pr(\overline{B}_{\epsilon(z)}(z))\bigr) 
  & \le  &
  \topdim\bigl(G \cdot \overline{B}_{\epsilon(z)}(z)\bigr).
  \label{dim(Z/G)_le_dim(Z):dim(pr(B))_le_dim(GB)}
\end{eqnarray}

Since $Z$ is proper, it is the countable union of compact subspaces
and hence contains a countable dense subset. This is equivalent to the
condition that $Z$ has a countable basis for its topology. Obviously
the same is true for $G\backslash Z$. We conclude
from~\cite[Theorem~9.1 in in Chapter~7.9 on page~302 and Exercise~9 in
Chapter~7.9 on page~315]{Munkres(1975)}
\begin{eqnarray}
  \topdim(Z) 
  & = & 
  \sup\bigl\{\topdim\bigl(G \cdot \overline{B}_{\epsilon(z)}(z)\bigr)\bigr\};
  \label{dim(Z)_is_sup}
  \\
  \topdim(G\backslash Z) 
  & = & 
  \sup\bigl\{\topdim\bigl(\pr(\overline{B}_{\epsilon(z)}(z))\bigr)\bigr\}.
  \label{dim(Z/G_is_sup}
\end{eqnarray}
Now Lemma~\ref{dim(Z/G)_le_dim(Z)} follows from
\eqref{dim(Z/G)_le_dim(Z):dim(pr(B))_le_dim(GB)}, \eqref{dim(Z)_is_sup} 
and~\eqref{dim(Z/G_is_sup}.
\end{proof}

In the sequel we will equip a simplicial complex with the weak
topology, i.e., a subset is closed if and only if its intersection
with any simplex $\sigma$ is a closed subset of $\sigma$. With this
topology a simplicial complex carries a canonical $CW$-structure.

Let $X$ be a $G$-space. We call a subset $U \subseteq X$ a
\emph{$\calfin$-set} if we have $gU \cap U \not= \emptyset \implies gU = U$ 
for every $g \in G$ and $G_U := \{g \in G \mid g \cdot U = U\}$
is finite. Let $\calu$ be a covering of $X$ by open
$\calfin$-subset. Suppose that $\calu$ is $G$-invariant, i.e., we have
$g \cdot U \in \calu$ for $g \in G$ and $U \in \calu$.  Define its
\emph{nerve} $\caln(\calu)$ to be the simplicial complex whose
vertices are the elements in $\calu$ and for which the pairwise
distinct vertices $U_0$, $U_1$, \ldots , $U_d$ span a $d$-simplex if
and only if $\bigcap_{i=0}^d U_i \not= \emptyset$. The action of $G$
on $X$ induces an action on $\calu$ and hence a simplicial action on
$\caln(\calu)$.  The isotropy group of any vertex is finite and hence
the isotropy group of any simplex is finite.  Let $\caln(\calu)'$ be
the barycentric subdivision. It inherits a simplicial $G$-action from
$\caln(\calu)$ such that for any $g \in G$ and any simplex $\sigma$
whose interior is denoted by $\sigma^{\circ}$ and which satisfies $g
\cdot \sigma^{\circ} \cap \sigma^{\circ} \not= \emptyset$ we have $gx = x$ 
for all $x \in \sigma^{\circ}$.  In particular $\caln(\calu)'$ is
a $G$-$CW$-complex and agrees as a $G$-space with $\caln(\calu)$.

\begin{lemma}\label{lem:map_to:G-CW_complex}
  Let $n$ be an integer with $n \ge 0$.  Let $X$ be a proper metric
  space whose topological dimension satisfies $\topdim(X) \le n$.
  Suppose that $G$ acts properly and isometrically on $X$.

  Then there exists a proper $n$-dimensional $G$-$CW$-complex $Y$
  together with a $G$-map $f \colon X \to Y$.
\end{lemma}
\begin{proof}
  Since the $G$-action is proper we can find for every $x \in X$ an
  $\epsilon(x) > 0$ such that for every $g \in G$ we have
  \begin{multline*}
    g \cdot \overline{B}_{2\epsilon(x)}(x) \cap
    \overline{B}_{2\epsilon(x)}(x)\not= \emptyset \Leftrightarrow g
    \cdot \overline{B}_{2\epsilon(x)}(x) = \overline{B}_{2\epsilon(x)}
    (x)
    \\
    \Leftrightarrow g \cdot B_{2\epsilon(x)}(x) = B_{2\epsilon(x)}(x)
    \Leftrightarrow g \cdot B_{\epsilon(x)}(x) = B_{\epsilon(x)}(x)
    \Leftrightarrow g \in G_x.
  \end{multline*}
  We can arrange that $\epsilon(gx) = \epsilon(x)$ for $g \in G$ and
  $x \in X$ holds.  We obtain a covering of $X$ by open
  $\calfin$-subsets $\bigl\{B_{\epsilon(x)}(x) \mid x \in X\bigr\}$.
  Let $\pr \colon X \to G\backslash X$ be the canonical projection.
  We obtain an open covering of $G\backslash X$ by
  $\bigl\{\pr\bigl(B_{\epsilon(x)}(x)\bigr) \mid x \in X\bigr\}$.
  Since $\topdim(X) \le n$ by assumption and $G$ acts properly on $X$,
  we get $\topdim(G\backslash X) \le n$ from
  Lemma~\ref{dim(Z/G)_le_dim(Z)}.  Since $G$ acts properly and
  isometrically on $X$, the quotient $G\backslash X$ inherits a metric
  from $X$. Hence $G\backslash X$ is paracompact by Stone's theorem
  (see~\cite[Theorem~4.3 in Chap.~6.3 on page~256]{Munkres(1975)}) and
  in particular normal.  By~\cite[Theorem~3.5 on
  page~211]{Dowker(1947map)} we can find a locally finite open
  covering $\calu$ of $G\backslash X$ such that $\dim(\calu) \le n$
  and $\calu$ is a refinement of 
  $\bigl\{\pr(B_{\epsilon(x)}(x)) \mid   x \in X\bigr\}$.  
  For each $U \in \calu$ choose $x(U) \in X$ with 
  $U \subseteq \pr\bigl(B_{\epsilon(U)}(x(U)\bigr)$.  Define the index
  set
$$J = \bigl\{(U,\overline{g}) \mid U \in \calu, \overline{g} \in G/G_{x(U)}\bigr\}.$$
For $(U,\overline{g}) \in J$ define an open $\calfin$-subset of $X$ by
$$V_{U,\overline{g}} := \pr^{-1}(U) \cap g \cdot B_{2\epsilon(x(U))}\bigl(x(U)\bigr).$$
Obviously this is well-defined, i.e., the choice of $g \in \overline{g}$ 
does not matter, and we have $\pr(V_{U,\overline{g}}) \subseteq U$ 
and $V_{U,\overline{g}} \subseteq g \cdot B_{2\epsilon(x(U))}\bigl(x(U)\bigr)$.

Consider the collection of subsets of $X$
$$\calv = \bigl\{V_{U,\overline{g}} \mid (U,\overline{g}) \in J\}.$$
This is a $G$-invariant covering of $X$ by open $\calfin$-subsets.
Its dimension satisfies
$$\dim(\calv) \le \dim(\calu) \le n$$
since for $U \in \calu$, $\overline{g_1}, \overline{g_2} \in
G/G_{x(U)}$ we have
$$V_{U,\overline{g_1}} \cap V_{U,\overline{g_2}} \not= \emptyset 
\implies g_1 \cdot B_{2\epsilon(x(U))}\bigl(x(U)\bigr) \cap g_2 \cdot
B_{2\epsilon(x(U))}\bigl(x(U)\bigr) \implies \overline{g_1} =
\overline{g_2}.$$

Since $\calu$ is locally finite and $G\backslash X$ is paracompact, we
can find a locally finite partition of unity $\bigl\{e_U \colon
G\backslash X \to [0,1] \mid U \in \calu\bigr\}$ which is subordinate
to $\calu$, i.e., $\sum_{U \in \calu} e_U = 1$ and $\supp(e_U) \subset U$ 
for every $U \in \calu$.  Fix a map $\chi \colon [0,\infty) \to [0,1]$ 
satisfying $\chi^{-1}(0) = [1,\infty)$.  Define for
$(U,\overline{g}) \in J$ a function
$$\phi_{U,\overline{g}} \colon X \to [0,1], \quad 
y \mapsto e_U(\pr(y)) \cdot
\chi\bigl(d_X(y,gx(U))/\epsilon(x(U))\bigr).$$

Consider $y \in X$. Since $\calu$ is locally finite and $G\backslash X$ is locally compact, 
we can find an open neighborhood $T$ of $\pr(y)$ such that $\overline{T}$ meets only finitely
many elements of $\calu$.  Choose an open neighborhood $W_0$ of $y$
such that $\overline{W_0}$ is compact.  Define an open neighborhood of
$y$ by
$$W := W_0 \cap \pr^{-1}(T).$$
Since $\overline{W_0}$ is compact, $\overline{W}$ is compact.  Since
$G$ acts properly, there exists for a given $U \in \calu$ only
finitely many elements $g \in G$ with $\overline{W} \cap g \cdot
B_{\epsilon(x(U))}(x(U)) \not= \emptyset$.  Since $\overline{T}$ meets only
finitely elements of $\calu$, the set
$$J_W := \bigl\{(U,\overline{g}) \in J\mid 
\overline{W} \cap g\cdot B_{\epsilon(x(U))}(x(U)) \cap \pr^{-1}(U)
\not= \emptyset\bigr\}$$ is finite. Suppose $\phi_{U,\overline{g}}(z) > 0$ 
for $(U,\overline{g}) \in J$ and $z \in W$.  We conclude 
$z \in \pr^{-1}(U) \cap g \cdot B_{\epsilon(x(U))}(x(U))$ and hence
$(U,\overline{g}) \in J_W$. Thus we have shown that the collection
$\bigl\{\phi_{U,\overline{g}} \mid (U,\overline{g}) \in J\bigr\}$ is
locally finite.

We conclude that the map
$$\sum_{(U,\overline{g}) \in J} \phi_{U,\overline{g}} \colon X \to [0,1], \quad 
y \mapsto \sum_{(U,\overline{g}) \in J} e_U(\pr(y)) \cdot
\chi\bigl(d_X(y,gx(U))/\epsilon(x(U))\bigr)$$ is well-defined and
continuous.  It has always a value greater than zero since for every
$y \in X$ there exists $U \in \calu$ with $e_U(\pr(y)) > 0$, the set
$\pr^{-1}(U)$ is contained in $\bigcup_{g \in G} g \cdot B_{\epsilon(U)}(x(U))$
and $\chi^{-1}(0) = [1,\infty)$.  Define for $(U,\overline{g}) \in J$
a map
$$\psi_{U,\overline{g}} \colon X \to [0,1], \quad 
y \mapsto \frac{\phi_{U,\overline{g}}(y)}{\sum_{(U,\overline{g}) \in
    J} \phi_{U,\overline{g}}(y)}.$$ 
We conclude that
$$\begin{array}{lcll}
  \sum_{(U,\overline{g}) \in J} \psi_{U,\overline{g}}(y) 
  & = & 1 & \quad \text{for}\; y \in X;
  \\
  \psi_{U,\overline{g}}(hy) 
   & = & \psi_{U,\overline{h^{-1}g}}(y) &
  \quad \text{for}\; h \in G, y \in Y  \; \text{and} \; (U,\overline{g}) \in J;
  \\
  \supp(\psi_{U,\overline{g}}) &\subseteq & V_{U,\overline{g}} &
  \quad \text{for} \; (U,\overline{g}) \in J,
\end{array}
$$
and the collection $\bigl\{\psi_{U,\overline{g}} \mid (U,\overline{g})
\in J\bigr\}$ is locally finite.  Define the desired proper
$n$-dimensional $G$-$CW$-complex to be the nerve $Y := \caln(\calv)$.
Define a map by
$$f \colon X \to \caln(\calv), \quad y \mapsto \sum_{(U,\overline{g}) \in J}
\psi_{U,\overline{g}}(y) \cdot V_{U,\overline{g}}.$$ It is
well-defined since for $y \in X$ the simplices $V_{U,\overline{g}}$
for which $\psi_{U,\overline{g}}(y) \not= 0$ holds span a simplex
because $y \in X$ with $\psi_{U,\overline{g}}(y) \not= 0$ belongs to
$V_{U,\overline{g}}$ and hence the intersection of the sets
$V_{U,\overline{g}}$ for which $\psi_{U,\overline{g}}(y) \not= 0$
holds contains $y$ and hence is non-empty. The map $f$ is continuous
since $\bigl\{\psi_{U,\overline{g}} \mid (U,\overline{g}) \in
J\bigr\}$ is locally finite.  It is $G$-equivariant by the following
calculation for $h \in G$ and $y \in Y$:
\begin{eqnarray*}
  f(hy) 
  & = & 
  \sum_{(U,\overline{g}) \in J} \psi_{U,\overline{g}}(hy) \cdot V_{U,\overline{g}}
  \\
  & = & 
  \sum_{(U,\overline{g}) \in J} \psi_{U,\overline{hg}}(hy) \cdot V_{U,\overline{hg}}
  \\
  & = & 
  \sum_{(U,\overline{g}) \in J} \psi_{U,\overline{h^{-1}hg}}(y) \cdot V_{U,\overline{hg}}
  \\
  & = & 
  \sum_{(U,\overline{g}) \in J} \psi_{U,\overline{g}}(y) \cdot h \cdot V_{U,\overline{g}}
  \\
  & = & 
  h \cdot \sum_{(U,\overline{g}) \in J} \psi_{U,\overline{g}}(y) \cdot  V_{U,\overline{g}}
  \\
  & = & 
  h \cdot f(y).
\end{eqnarray*}
\end{proof}

\begin{lemma} \label{lem:domination_and_dimension} Let $X$ and $Y$ be
  $G$-$CW$-complexes.  Let $i \colon X \to Y$ and $r \colon Y \to X$
  be $G$-maps such that $r \circ i$ is $G$-homotopic to the
  identity. Consider an integer $d \ge 3$.  Suppose that $Y$ has
  dimension $\le d$.

  Then $X$ is $G$-homotopy equivalent to a $G$-$CW$-complex $Z$ of
  dimension $\le d$.
\end{lemma}
\begin{proof}
  By the Equivariant Cellular Approximation Theorem
  (see~\cite[Theorem~II.2.1 on page~104]{Dieck(1987)}) we can assume
  without loss of generality that $i$ and $r$ are cellular. Let
  $\cyl(r)$ be the mapping cylinder. Let $k \colon Y \to \cyl(r)$ be
  the canonical inclusion and $p \colon \cyl(r) \to X$ be the
  canonical projection. Then $p$ is a $G$-homotopy equivalence and $p
  \circ k = r$.  Let $Z$ be the union of the $2$-skeleton of $\cyl(r)$
  and $Y$. This is a $G$-$CW$-subcomplex of $\cyl(r)$ and $\cyl(r)$ is
  obtained from $Z$ by attaching equivariant cells of dimension $\ge
  3$.  Hence the map $p|_Z \colon Z \to X$ has the property that it
  induces on every fixed point set a $2$-connected map. Let $j \colon
  X \to Z$ be the composite of $i \colon X \to Y$ with the obvious
  inclusion $Y \to Z$. Then $p|_Z \circ j = p \circ k \circ i = r
  \circ i$ is $G$-homotopy equivalent to the identity and the
  dimension of $Z$ is still bounded by $d$ since we assume $d \ge
  3$. Hence we can assume in the sequel that $r^H \colon Y^H \to X^H$
  is $2$-connected for all $H \subseteq G$, otherwise replace $Y$ by
  $Z$, $i$ by $j$ and $r$ by $p|_Z$.

  We want to apply~\cite[Proposition~14.9 on page~282]{Lueck(1989)}.
  Here the assumption $d \ge 3$ enters. Hence it suffices to show that
  the cellular $\IZ\Pi(G,X)$-chain complex $C_*^c(X)$ is
  $\IZ\Pi(G,X)$-chain homotopy equivalent to a $d$-dimensional
  $\IZ\Pi(G,X)$-chain complex. By~\cite[Proposition~11.10 on
  page~221]{Lueck(1989)} it suffices to show that the cellular
  $\IZ\Pi(G,X)$-chain complex $C_*^c(X)$ is dominated by a
  $d$-dimensional $\IZ\Pi(G,X)$-chain complex. This follows from the
  geometric domination $(Y,i,r)$ by passing to the cellular chain
  complexes over the fundamental categories since $r$ and hence also
  $i$ induce equivalences between the fundamental categories because
  $r^H \colon Y^H \to X^H$ is $2$-connected for all $H \subseteq G$
  and $r \circ i \simeq_G \id_X$.

\end{proof}

The condition $d \ge 3$ is needed since we want to argue first with
the cellular $\IZ\Or(G)$-chain complex and then transfer the statement
that it is $d$-dimensional to the statement that the underlying
$G$-$CW$-complex is $d$-dimensional. The condition $d \ge 3$ enters
for analogous reasons in the classical proof of the theorem that the
existence of a $d$-dimensional $\IZ G$-projective resolution for the
trivial $\IZ G$-module $\IZ$ implies the existence of a
$d$-dimensional model for $BG$ (see~\cite[Theorem~7.1 in
Chapter~VIII.7 on page~205]{Brown(1982)}).

\begin{theorem} \label{the:d-dimensional_model_for_eub(G)_from_jub(G)}
  Let $G$ be a discrete group. Then

  \begin{enumerate}

  \item \label{the:d-dimensional_model_for_eub(G)_from_jub(G):G-homotopy_equivalence}
    There is a $G$-homotopy equivalence $\jub{G} \to \eub{G}$;

  \item \label{the:d-dimensional_model_for_eub(G)_from_jub(G):dimension_from_J_to_E}
    Suppose that there is a model for $\jub{G}$ which is a metric
    space such that the action of $G$ on $\jub{G}$ is
    isometric. Consider an integer $d$ with $d = 1$ or $d \ge 3$.
    Suppose that the topological dimension $\topdim(\jub{G}) \le d$.

    Then there is a $G$-$CW$-model for $\eub{G}$ of dimension $\le d$;

  \item \label{the:d-dimensional_model_for_eub(G)_from_jub(G):dimension_from_E_to_J}
    Let $d$ be an integer $d \ge 0$. Suppose that there is a
    $G$-$CW$-model for $\eub{G}$ with $\dim(\eub{G}) \le d$ such that
    $\eub{G}$ after forgetting the group action has countably many
    cells.
 
    Then there exists a model for $\jub{G}$ with $\topdim(\jub{G}) \le
    d$.

  \end{enumerate}
\end{theorem}
\begin{proof}~\ref{the:d-dimensional_model_for_eub(G)_from_jub(G):G-homotopy_equivalence}
  This is proved in~\cite[Lemma~3.3 on page 278]{Lueck(2005s)}.
  \\[1mm]~\ref{the:d-dimensional_model_for_eub(G)_from_jub(G):dimension_from_J_to_E}
  Choose a $G$-homotopy equivalence $i \colon \eub{G} \to \jub{G}$.
  From Lemma~\ref{lem:map_to:G-CW_complex} we obtain a $G$-map 
  $f \colon \jub{G} \to Y$ to a proper $G$-$CW$-complex of dimension 
  $\le d$. By the universal properly of $\eub{G}$ we can find a $G$-map 
  $h \colon Y \to \eub{G}$ and the composite $h \circ f \circ i$ is
  $G$-homotopic to the identity on $\eub{G}$.

  Suppose $d \ge 3$. We conclude from
  Lemma~\ref{lem:domination_and_dimension} that $\eub{G}$ is
  $G$-homotopy equivalent to a $G$-$CW$-complex of dimension $\le d$.

  Suppose $d = 1$. By Dunwoody~\cite[Theorem~1.1]{Dunwoody(1979)} it
  suffices to show that the rational cohomological dimension of $G$
  satisfies $\cd_{\IQ}(G)\le 1$. Hence we have to show for any $\IQ
  G$-module $M$ that $\Ext^n_{\IQ G}(\IQ,M\bigr) = 0$ for $n \ge 2$,
  where $\IQ$ is the trivial $\IQ G$-module.  Since all isotropy
  groups of $\eub{G}$ and $Y$ are finite, their cellular $\IQ G$-chain
  complexes are projective. Since $\eub{G}$ is contractible,
  $C_*(\eub{G};\IQ)$ is a projective $\IQ G$-resolution and hence
$$\Ext^n_{\IQ G}(\IQ,M\bigr) \cong H^n\bigl(\hom_{\IQ G}(C_*(\eub{G};\IQ),M)\bigr).$$
Since $h \circ f \circ i \simeq_G \id_{\eub{G}}$, the $\IQ$-module
$H^n\bigl(\hom_{\IQ G}(C_*(\eub{G};\IQ),M)\bigr)$ is a direct summand
in the $\IQ$-module $H^n\bigl(\hom_{\IQ G}(C_*(Y;\IQ),M)\bigr)$.
Since $Y$ is $1$-dimensional by assumption, 
$H^n\bigl(\hom_{\IQ G}(C_*(Y;\IQ),M)\bigr)$ vanishes for $n \ge 2$. This implies that
$\Ext^n_{\IQ G}(\IQ,M\bigr)$ vanishes for $n \ge 2$.  \\[1mm]%
\ref{the:d-dimensional_model_for_eub(G)_from_jub(G):dimension_from_E_to_J}
Using the equivariant version of the simplicial approximation theorem
and the fact that changing the $G$-homotopy class of attaching maps
does not change the $G$-homotopy type, one can find a simplicial
complex $X$ with simplicial $G$-action which is $G$-homotopy
equivalent to $\eub{G}$, satisfies $\dim(X) = \dim(\eub{G})$ and has
only countably many simplices.  Hence the barycentric subdivision $X'$
is a simplicial complex of dimension $\le d$ with countably many
simplices and carries a $G$-$CW$-structure.  The latter implies that
$X'$ is a $G$-$CW$-model for $\eub{G}$ and hence also a model for
$\jub{G}$.  Since the dimension of a simplicial complex with countably
many simplices is equal to its topological dimension, we conclude
$\topdim(X') = \dim(X) = \dim(\eub{G}) \le d$.
\end{proof}


\typeout{--------------------   Section 3   --------------------------}

\section{The passage from finite to virtually cyclic groups}
\label{sec:The_passage_from_finite_to_virtually_cyclic_groups}

In~\cite{Lueck-Weiermann(2007)} it is described how one can construct
$\edub{G}$ from $\eub{G}$.  In this section we want to make this
description more explicit under the following condition

\begin{condition} \label{cond:(C)} We say that $G$ satisfies condition
  (C) if for every $g,h \in G$ with $|h| = \infty$ and $k,l \in \IZ$
  we have
$$gh^kg^{-1} = h^l \implies |k| = |l|.$$
\end{condition}

Let $\calicyc$ be the set of infinite cyclic subgroup $C$ of $G$.
This is not a family since it does not contain the trivial subgroup.
We call $C,D \in \calicyc$ \emph{equivalent} if $|C \cap D| =
\infty$. One easily checks that this is an equivalence relation on
$\calicyc$. Denote by $[\calicyc]$ the set of equivalence classes and
for $C \in \calicyc$ by $[C]$ its equivalence class. Denote by
$$N_GC := \{g \in G \mid gCg^{-1} = C\}$$
the \emph{normalizer} of $C$ in $G$. Define for $[C] \in [\calicyc]$ a
subgroup of $G$ by
$$N_G[C] := \bigl\{g \in G \;\big|\; |gCg^{-1} \cap C| = \infty\bigr\}.$$
One easily checks that this is independent of the choice of $C \in
[C]$.  Actually $N_G[C]$ is the isotropy of $[C]$ under the action of
$G$ induced on $[\calicyc]$ by the conjugation action of $G$ on
$\calicyc$.

\begin{lemma}
  \label{lem:N_G[C]_is_colimC_Gk!C}
  Suppose that $G$ satisfies Condition (C) (see~\ref{cond:(C)}).
  Consider $C \in \calicyc$.

  Then obtain a nested sequence of subgroups
$$N_GC \subseteq N_G2!C \subseteq N_G3!C \subseteq N_G4!C \subseteq \cdots $$
where $k!C$ is the subgroup of $C$ given by $\{h^{k!} \mid h \in C\}$,
and we have
$$N_G[C] = \bigcup_{k \ge 1} N_Gk!C.$$
\end{lemma}
\begin{proof}
  Since every subgroup of a cyclic group is characteristic, we obtain
  the nested sequence of normalizers $N_GC \subseteq N_G2!C \subseteq
  N_G3!C \subseteq N_G4!C \subseteq \cdots $.

  Consider $g \in N_G[C]$. Let $h$ be a generator of $C$.  Then there
  are $k,l \in \IZ$ with $gh^kg^{-1} = h^l$ and $k,l \not=
  0$. Condition (C) implies $k = \pm l$. Hence 
  $g \in N_G\langle h^k \rangle \subseteq N_Gk!C$. 
  This implies $N_G[C] \subseteq \bigcup_{k \ge 1} N_Gk!C$. 
  The other inclusion follows from the fact that for
  $g \in N_Gk!C$ we have $k!C \subseteq gCg^{-1} \cap C$.
\end{proof}

Fix $C \in \calicyc$. Define a family of subgroups of $N_G[C]$ by
\begin{multline}
  \calg_G(C) := \bigl\{H \subseteq N_G[C] \mid [H : (H \cap C)] <
  \infty\bigr\} \\ \cup \bigl\{H \subseteq N_G[C] \mid |H| <
  \infty\bigr\}.
  \label{calg_G(C)}
\end{multline}
Notice that $\calg_G(C)$ consists of all finite subgroups of $N_G[C]$
and of all virtually cyclic subgroups of $N_G[C]$ which have an
infinite intersection with $C$.  Define a quotient group of $N_GC$ by
$$W_GC := N_GC/C.$$

\begin{lemma} \label{lem:dim(edub(G)_bounded_from_above} Let $n$ be an
  integer. Suppose that $G$ satisfies Condition (C)
  (see~\ref{cond:(C)}).  Suppose that there exists a $G$-$CW$-model
  for $\eub{G}$ with $\dim(\eub{G}) \le n$ and for every $C \in
  \calicyc$ there exists a $W_GC$-$CW$-model for $\eub{W_GC}$ with
  $\dim(\eub{W_GC}) \le n$.

  Then there exists a $G$-$CW$-model for $\edub{G}$ with
  $\dim(\edub{G}) \le n+1$.
\end{lemma}
\begin{proof}
  Because of~\cite[Theorem~2.3 and Remark~2.5]{Lueck-Weiermann(2007)}
  it suffices to show for every $C \in \calicyc$ that there is a
  $N_G[C]$-model for $\EGF{N_G[C]}{\calg_G(C)}$ with
  \begin{eqnarray}
    \dim(\EGF{N_G[C]}{\calg_G(C)}) & \le & n+1.
    \label{lem:dim(edub(G)_bounded_from_above:dim}
  \end{eqnarray}
  Because of Lemma~\ref{lem:N_G[C]_is_colimC_Gk!C} we have
$$N_G[C] = \colim_{k \to \infty}  N_Gk!C.$$
We conclude~\eqref{lem:dim(edub(G)_bounded_from_above:dim}
from~\cite[Lemma~4.2 and Theorem~4.3]{Lueck-Weiermann(2007)} since
every element $H \in \calg_G(C)$ is
finitely generated and hence lies already in $N_Gk!C$ for some $k > 0$,
by assumption there exists a $W_Gk!C$-$CW$-model for $\eub{W_Gk!C}$ with
$\dim(\eub{W_Gk!C}) \le n$, and $\res_{N_Gk!C \to W_Gk!C} \eub{W_Gk!C}$ is 
$\EGF{N_Gk!C}{\calg_G(C)|_{N_Gk!C)}}$.

\end{proof}

Now we are ready to prove
Theorem~\ref{the:dim(edub(G))_for_Cat(0)-group}.
\begin{proof}[Proof of Theorem~\ref{the:dim(edub(G))_for_Cat(0)-group}]~%
\ref{the:dim(edub(G))_for_Cat(0)-group:fin} Consider an integer $d
  \in \IZ$ with $d = 1$ or $d \ge 3$ such that $d \ge \topdim(X)$.
  The space $X$ is a model for $\jub{G}$ by~\cite[Corollary~2.8 in
  II.2. on page~178]{Bridson-Haefliger(1999)}.  We conclude from
  Theorem~\ref{the:d-dimensional_model_for_eub(G)_from_jub(G)}~%
\ref{the:d-dimensional_model_for_eub(G)_from_jub(G):dimension_from_J_to_E}
  that there is a $d$-dimensional model for $\eub{G}$.  
  \\[2mm]~\ref{the:dim(edub(G))_for_Cat(0)-group:vcyc} 
  We will use in the
  proof some basic facts and notions about isometries of proper
  complete $\CAT(0)$-spaces which can be found
  in~\cite[Chapter~II.6]{Bridson-Haefliger(1999)}.

  The group $G$ satisfies condition (C) by the following argument.
  Suppose that $gh^kg^{-1} = h^l$ for $g,h \in G$ with $|h| = \infty$
  and $k,l \in \IZ$.  The isometry $l_h \colon X \to X$ given by
  multiplication with $h$ is a hyperbolic isometry since it has no
  fixed point and is by assumption semisimple.  We obtain for the
  translation length $L(h)$ which is a real number satisfying $L(h) > 0$
  $$k \cdot L(h) = L(h^k) = L(gh^kg^{-1}) = L(h^l) = l \cdot L(h).$$
  This implies $k = l$.

Let $C \subseteq G$ be any infinite cyclic subgroup.  Choose a
generator $g \in C$.  The isometry $l_g \colon X \to X$ given by
multiplication with $g$ is a hyperbolic isometry. Let $\Min(g)\subset
X$ be the the union of all axes of $g$.  Then $\Min(g)$ is a closed
convex subset of $X$. There exists a closed convex subset $Y(g)
\subseteq X$ and an isometry
$$\alpha \colon \Min(g) \xrightarrow{\cong} Y(g) \times \IR.$$
The space $\Min(G)$ is $N_GC$-invariant since for each $h \in N_GC$ we
have $hgh^{-1} = g$ or $hgh^{-1} = g^{-1}$ and hence multiplication
with $h$ sends an axis of $g$ to an axis of $g$. The $N_GC$-action
induces a proper isometric $W_GC$-action on $Y(g)$. These claims
follow from~\cite[Theorem~6.8 in~II.6 on page~231 and Proposition~6.10
in~II.6 on page~233]{Bridson-Haefliger(1999)}.  The space $Y(g)$
inherits from $X$ the structure of a $\CAT(0)$-space and satisfies
$\topdim(Y(g)) \le \topdim(X)$. Hence $Y(g)$ is a model for
$\jub{W_GC}$ with $\topdim(Y(g)) \le \topdim(X)$
by~\cite[Corollary~2.8 in II.2. on
page~178]{Bridson-Haefliger(1999)}. We conclude from
Theorem~\ref{the:d-dimensional_model_for_eub(G)_from_jub(G)}~%
\ref{the:d-dimensional_model_for_eub(G)_from_jub(G):dimension_from_J_to_E}
that there is a $d$-dimensional model for $\eub{W_GC}$ for every
infinite cyclic subgroup $C \subseteq G$. Now
Theorem~\ref{the:dim(edub(G))_for_Cat(0)-group} follows from
Lemma~\ref{lem:dim(edub(G)_bounded_from_above}.
\end{proof}

Finally we prove Corollary~\ref{cor:manifold}.
\begin{proof}[Proof of Corollary~\ref{cor:manifold}]
  A complete Riemannian manifold $M$ with non-negative sectional
  curvature is a $\CAT(0)$-space (see~\cite[Theorem~IA.6 on page~173
  and Theorem~II.4.1 on page~193]{Bridson-Haefliger(1999)}.)  Since
  $G$ is virtually torsionfree, we can find a subgroup $G_0$ of finite
  index in $G$ such that $G_0$ is torsionfree and acts orientation
  preserving on $M$. Hence $G_0\backslash M$ is a closed orientable
  manifold of dimension $n$.  Hence $H_n(M;\IZ) = H_n(BG;\IZ) \not=
  0$. This implies that every $CW$-model $BG_0$ has at least dimension
  $n$. Since the restriction of $\eub{G}$ to $G_0$ is a
  $G_0$-$CW$-model for $EG_0$, we conclude $\hdim(\eub{G}) \ge
  n$. Since $M$ with the given $G_0$-action is a $G$-$CW$-model for
  $\eub{G}$ (see~\cite[Theorem 4.15]{Abels(1978)}), we conclude
$$\hdim(\eub{G}) = n = \topdim(M).$$
If $n \not = 2$, we conclude $\hdim(\edub{G}) \le n+1$ from
Theorem~\ref{the:dim(edub(G))_for_Cat(0)-group}. Since $\hdim(\eub{G})
\le 1 + \hdim(\edub{G})$ holds for all groups $G$
(see~\cite[Corollary~5.4]{Lueck-Weiermann(2007)}), we get
$$n-1 \le \hdim(\edub{G}) \le n+1$$
provided that $n \not= 2$.

Suppose $n = 2$. If $G_0$ is a torsionfree subgroup of finite index in
$G$, then $G_0\backslash X$ is a closed $2$-dimensional manifold with
non-negative sectional curvature. Hence $G_0$ is $\IZ^2$ or
hyperbolic.  This implies that $G$ is virtually $\IZ^2$ or
hyperbolic. Hence $\hdim(\edub{G}) \in \{2,3\}$ 
by~\cite[Example~5.21]{Lueck-Weiermann(2007)}
in the first case and 
by~\cite[Theorem~3.1,
Example~3.6, Theorem~5.8~(ii)]{Lueck-Weiermann(2007)}
or~\cite[Proposition~6, Remark~7 and Proposition~8]{Juan-Pineda-Leary(2006)} in the 
second case.

\end{proof}


\begin{thebibliography}{10}

\bibitem{Abels(1978)}
H.~Abels.
\newblock A universal proper ${G}$-space.
\newblock {\em Math. Z.}, 159(2):143--158, 1978.

\bibitem{Bredon(1972)}
G.~E. Bredon.
\newblock {\em Introduction to compact transformation groups}.
\newblock Academic Press, New York, 1972.
\newblock Pure and Applied Mathematics, Vol. 46.

\bibitem{Bridson-Haefliger(1999)}
M.~R. Bridson and A.~Haefliger.
\newblock {\em Metric spaces of non-positive curvature}.
\newblock Springer-Verlag, Berlin, 1999.
\newblock Die Grundlehren der mathematischen Wissenschaften, Band 319.

\bibitem{Brown(1982)}
K.~S. Brown.
\newblock {\em Cohomology of groups}, volume~87 of {\em Graduate Texts in
  Mathematics}.
\newblock Springer-Verlag, New York, 1982.

\bibitem{Davis-Lueck(1998)}
J.~F. Davis and W.~L{\"u}ck.
\newblock Spaces over a category and assembly maps in isomorphism conjectures
  in ${K}$- and ${L}$-theory.
\newblock {\em $K$-Theory}, 15(3):201--252, 1998.

\bibitem{Dowker(1947map)}
C.~H. Dowker.
\newblock Mapping theorems for non-compact spaces.
\newblock {\em Amer. J. Math.}, 69:200--242, 1947.

\bibitem{Dunwoody(1979)}
M.~J. Dunwoody.
\newblock Accessibility and groups of cohomological dimension one.
\newblock {\em Proc. London Math. Soc. (3)}, 38(2):193--215, 1979.

\bibitem{Juan-Pineda-Leary(2006)}
D.~Juan-Pineda and I.~J. Leary.
\newblock On classifying spaces for the family of virtually cyclic subgroups.
\newblock In {\em Recent developments in algebraic topology}, volume 407 of
  {\em Contemp. Math.}, pages 135--145. Amer. Math. Soc., Providence, RI, 2006.

\bibitem{Lueck(1989)}
W.~L{\"u}ck.
\newblock {\em Transformation groups and algebraic ${K}$-theory}, volume 1408
  of {\em Lecture Notes in Mathematics}.
\newblock Springer-Verlag, Berlin, 1989.

\bibitem{Lueck(2005s)}
W.~L{\"u}ck.
\newblock Survey on classifying spaces for families of subgroups.
\newblock In {\em Infinite groups: geometric, combinatorial and dynamical
  aspects}, volume 248 of {\em Progr. Math.}, pages 269--322. Birkh\"auser,
  Basel, 2005.

\bibitem{Lueck-Weiermann(2007)}
W.~L\"uck and M.~Weiermann.
\newblock On the classifying space of the family of virtually cyclic subgroups.
\newblock Preprintreihe SFB 478 --- Geometrische Strukturen in der Mathematik,
  Heft 453, M\"unster, arXiv:math.AT/0702646v2, to appear in the Proceedings in
  honour of Farrell and Jones in Pure and Applied Mathematic Quarterly, 2007.

\bibitem{Munkres(1975)}
J.~R. Munkres.
\newblock {\em Topology: a first course}.
\newblock Prentice-Hall Inc., Englewood Cliffs, N.J., 1975.

\bibitem{Dieck(1972)}
T.~tom Dieck.
\newblock Orbittypen und \"aquivariante {H}omologie. {I}.
\newblock {\em Arch. Math. (Basel)}, 23:307--317, 1972.

\bibitem{Dieck(1987)}
T.~tom Dieck.
\newblock {\em Transformation groups}.
\newblock Walter de Gruyter \& Co., Berlin, 1987.

\end{thebibliography}
\end{document}